\documentclass{article}
\usepackage[utf8]{inputenc}
\usepackage{amsmath}
\allowdisplaybreaks[3]
\usepackage[english]{babel}
\usepackage{csquotes}
\usepackage[backend=biber,style=apa,sorting=apa,natbib=true]{biblatex}
\DeclareLanguageMapping{english}{english-apa}
\let\cite\textcite
\usepackage{authblk}
\usepackage{booktabs}
\usepackage{multirow}
\usepackage{amssymb}
\usepackage{graphicx}
\usepackage{pgfplots}
\usepackage{pgfplotstable} 
\usepackage{subfigure}
\usepackage{pifont}

\usepackage{xcolor}

\usepackage[ruled,linesnumbered]{algorithm2e}

\usepackage{amsthm}
\newtheorem{theorem}{Theorem}
\newtheorem{lemma}{Lemma}
\newtheorem{proposition}{Proposition}
\newtheorem{definition}{Definition}

\usepackage{multirow}
\usepackage{enumitem}

\usepackage[skip=0.1\baselineskip]{caption}

\usepackage[title]{appendix}
\usepackage{hyperref}

\title{The optimality of $(s, S)$ hiring policies in workforce planning with fixed recruitment costs and binomial turnover}
\author[a]{Zhen Chen\thanks{Corresponding author: chen.zhen5526@gmail.com}}
\author[b]{Roberto Rossi\thanks{roberto.rossi@ed.ac.uk}}
\author[b]{Belen Martin-Barragan\thanks{belen.martin@ed.ac.uk}}
\author[c]{S. Armagan Tarim\thanks{armtar@yahoo.com}}
\affil[a]{Business School, Brunel University London, Uxbridge UB8 3PH,  United Kingdom}
\affil[b]{Business School, University of Edinburgh, Edinburgh EH8 9JS, United Kingdom}
\affil[c]{ Institute Of Informatics, Hacettepe University, Ankara 06800, Turkey}

\date{}
\pgfplotsset{compat=1.18} 

\usepackage[margin=1in]{geometry}   
\usepackage{setspace}               
\begin{document}

\maketitle

\begin{abstract}
We study a finite-horizon nonstationary workforce planning problem in which staff turnover in each period follows a binomial distribution whose parameters depend on the post-hiring workforce level. The model incorporates a fixed hiring cost that is incurred whenever recruitment occurs, regardless of the number of employees hired; this represents explicit (e.g. administration costs) or implicit (e.g. team disruption) factors that are commonplace in workforce planning. The objective is to minimise the expected total cost, including recruitment, salary, and shortage costs, where deviations below period-specific staffing requirements are penalised. To analyse this stochastic dynamic program with decision-dependent transition probabilities, we establish the discrete convexity of the variable single-period cost (the sum of expected salary and penalty costs) and the $K$-convexity of the expected total cost in presence of endogenous uncertainty. We then show that the optimal hiring policy exhibits an $(s, S)$-type structure: when the workforce level in a given period falls below a threshold $s$, staff are hired up to level $S$; otherwise, no hiring occurs. To compute approximate policy parameters, we develop a mixed-integer linear programming heuristic based on a piecewise approximation. In our computational study, we demonstrate that our approach achieves fast computation and small optimality gaps on the test instances.

\end{abstract}


\section{Introduction}

Effective workforce planning aims to ensure that ``Securing the services of the right people, in the right place, at the right time'' \citep{taylor2005people}. In practice, however, organisations often alternate between shortage and excess. On Monday, 25 July 2022, the Guardian front page featured the headline: ```Greatest staffing crisis' in NHS history leaves patients at risk'' \citep{Denis2022}; only three years later, thousands of staff in NHS England faced job losses as a result of financial pressures and cost-cutting measures \citep{triggle2025}. Such episodes illustrate a central difficulty of workforce planning: organisations must choose staffing levels before future demand, turnover, and budgetary conditions are fully known, while later corrections through overtime, outsourcing, temporary labour, or layoffs may be costly and disruptive.

Staff turnover is a particularly important source of uncertainty. A large literature in management and applied psychology shows that employee turnover varies substantially across occupations, organisations, and labour-market contexts, and that it has important implications for staffing, productivity, and organisational performance \citep{Hausknecht2010,Park2013,hom2017one}. These factors render employee turnover highly {\em nonstationary}. Turnover is also costly: beyond the direct costs of recruiting, selecting, and training replacements, organisations may incur productivity losses, disruption to teams, and loss of firm-specific human capital \citep{Cascio2006,Tracey2008,Hancock2011}. In the workforce planning literature, such costs are often referred to as wastage, a key driver of recruitment and promotion decisions \citep{DeFeyter2016,Komarudin2015}. 

A distinctive feature of workforce planning is that, unlike physical goods, service capacity cannot be stored. Nevertheless, employees can play a role analogous to inventory in service systems: maintaining a workforce buffer reduces shortage risk but increases salary costs, whereas hiring too late may lead to service failures or expensive recourse. Early research already highlighted this trade-off by comparing stable workforce pools with hire-and-fire policies under fluctuating requirements \citep{Conrath1971}. Similar tensions arise in fire departments, hospitals, call centres, public agencies, and military organisations, where staffing decisions must be made before employee availability and workload are fully realised \citep{Fry2006,Green2013}.

Whenever an organisation hires staff, it may incur costs that are independent of the number of employees hired, including administration costs, recruitment-agency fees, vacancy advertisements, headhunting, screening, training set-up, and medical examinations \citep{hamermesh1989labor}. Additional fixed costs may be implicit: adding new staff may disrupt existing teams, affect productivity, or impose managerial effort \citep{Navarra2019}. If these costs are substantial, continuous just-in-time hiring may be suboptimal and batch recruitment may be preferable. We argue that considering these costs in a workforce planning model is important in order to reduce setup instability \citep{Kilic2011} and the associated nervousness of control action \citep{deKok1997}.

Our work hence builds on the respective inventory control tradition. A seminal result in this field is the optimality of the $(s,S)$ policy for the dynamic inventory problem \citep{Scarf1960}: one orders up to $S$ when inventory falls below $s$, and does not order otherwise. Empirical evidence from \cite{caballero1997aggregate} shows that employment adjustment is indeed often lumpy and discontinuous. This behaviour is naturally reminiscent of an $(s,S)$ control policy: organisations do not continuously adjust headcount, but rather initiate recruitment campaigns only when staffing falls sufficiently low. 

In classical inventory models, demand is exogenous and independent of the order-up-to level. In our setting, turnover is decision-dependent: after hiring up to level $y$, the number of employees who leave during the period follows a binomial distribution with $y$ trials, so transition probabilities depend on the hiring decisions. This represents a major difference from the classical inventory control setting. Increasing the workforce simultaneously expands service capacity and enlarges the population exposed to turnover risk. This feature resembles inventory models with stock-dependent demand or deterioration, where the random depletion of stock depends on the inventory level. In such models, the classical $K$-convexity arguments that are cornerstone to proving optimality of $(s,S)$ policies do not apply automatically, and may fail without additional structure. Our analysis shows that, for binomial turnover, this difficulty can be overcome: although the transition is decision-dependent, it preserves a suitable form of $K$-convexity.


This paper makes the following contributions.

\begin{itemize}
    \item We formulate a finite-horizon stochastic workforce planning problem with fixed recruitment costs and nonstationary binomial staff turnover. Unlike standard inventory models, the turnover distribution is decision-dependent because the number of potential leavers depends on the post-hiring workforce level.

    \item We prove the discrete convexity of the single-period expected salary and shortage-penalty cost. This property is nontrivial because the underlying binomial random variable depends on the staffing decision.

\item We show that \(K\)-convexity is preserved under the binomial propagation induced by workforce turnover. This is the key technical step: although the transition probability depends on the post-hiring workforce level, the binomial structure preserves enough convexity to recover an \((s,S)\) optimality result.

    \item We prove that the optimal hiring policy has an \((s,S)\) structure in every period. This result provides a rigorous operational justification for lumpy hiring under stochastic turnover and fixed recruitment costs.


    \item We develop a tractable mixed-integer linear programming approximation to compute the policy parameters, demonstrating computationally that it achieves small optimality gaps in short solution times.
\end{itemize}

The rest of this work is structured as follows. Section \ref{sec:literature} reviews the related literature. Section \ref{sec:problem-setting} describes the problem setting and formulates a stochastic dynamic programming model. Section \ref{sec:math-property} proves the optimal policy of $(s, S)$. Section \ref{sec:computation-milp} presents a MILP model for computing the policy parameters. A computational study and its results are detailed in Section \ref{sec:numerical-tests}. Finally, Section \ref{sec:conclusion} draws conclusions and outlines future research directions.

\section{Literature Review}\label{sec:literature}

Our paper is related to four streams of literature: inventory-inspired workforce planning, strategic workforce planning models, operational staffing and scheduling under uncertainty, and models with fixed adjustment costs.

\textit{Inventory-inspired workforce planning.}
There is a close connection between inventory control and workforce planning \citep[see e.g.,][]{gans2002managing}. The analogy is that employees provide capacity in the same way that inventory provides availability: too little capacity causes shortages, whereas too much capacity generates holding or salary costs. Early work on workforce pooling used this analogy to compare stable workforce pools with hire-and-fire policies under fluctuating requirements \citep{Conrath1971}. \cite{gans2002managing} develop a Markov decision process for employee staffing with stochastic learning and turnover and show that, when fixed hiring costs are absent, the optimal policy is a state-dependent hire-up-to policy, analogous to a base-stock policy in inventory control. A related discrete-worker model is studied by \cite{ahn2005staffing}, who consider staffing decisions for heterogeneous workers with turnover and use supermodularity arguments to characterise structural properties of optimal policies. Our work differs in two key respects: we introduce a fixed recruitment cost, and we prove that the optimal policy becomes an $(s,S)$ policy rather than a base-stock policy. While this result is arguably not surprising, the analysis required to establish it is not a trivial extension of results found in the inventory control literature.

The decision-dependent transition in our model is related to inventory models with stock-dependent demand, deterioration, or spoilage. In these settings, inventory can be depleted by a mechanism whose distribution depends on the inventory level, an analogy already noted in workforce models with turnover \citep{iglehart1969multi,gans2002managing}. More recent inventory models study demand distributions that depend on displayed or ending inventory levels, leading to state-dependent stocking policies in newsvendor and periodic-review settings \citep{balakrishnan2008integrating,sapra2010much,yang2014dynamic,xue2017managing}. These papers are relevant because they relax the classical assumption that demand is independent of stock. However, they do not address fixed recruitment costs or workforce turnover, and their structural arguments do not directly yield an $(s,S)$ hiring policy under binomial decision-dependent wastage. In particular, stock-dependent demand undermines the $K$-convexity structure used in classical inventory theory. Conversely, our binomial turnover model creates a specific form of endogeneity --- between the post-hiring workforce level and the number of employees who leave --- for which we are able to prove that an appropriate $K$-convexity property is preserved.

A related analogy is with maintenance and repair systems \citep{Jardine2026-hi}: employee turnover corresponds to failures, while hiring corresponds to repair or replacement. This analogy is useful at a high level but limited technically. Maintenance models usually require tracking the age, condition, or lifetime distribution of individual machines, and failure distributions are often non-binomial. A binomial transition would be appropriate only under a restrictive memoryless one-period failure model with statistically identical units. In line with existing works in the literature, our workforce model abstracts from individual tenure or age and studies the aggregate headcount dynamics induced by binomial turnover. This abstraction is also practically relevant: individual-level workforce models require detailed employee histories, such as tenure, age, performance, and career trajectories, which may be difficult to obtain, sensitive from a human resource and privacy perspective, and costly to maintain. Aggregate headcount models therefore offer a more parsimonious and operationally acceptable representation for strategic and tactical workforce planning.

\textit{Strategic workforce planning models.}
A broader strategic workforce planning literature studies recruitment, promotion, training, retention, attrition, and scheduling decisions. The workforce planning model of \cite{Holt1960-xc} and its linear hiring rules are an early foundation. Recent surveys emphasize the diversity of modelling approaches --- including Markov models, mathematical programming, simulation, system dynamics, and empirical methods --- and classify workforce systems in terms of recruitment, attrition, promotion, training, retention, and scheduling \citep{Peck2025}. Much of this literature focuses on workforce heterogeneity and skill allocation, considering attributes such as age \citep{valls2009skilled}, experience \citep{firat2012improved}, technical knowledge \citep{corominas2012detailed}, licences \citep{krishnamoorthy2012algorithms}, nurse grades \citep{maenhout2013integrated}, and military ranks \citep{turan2021multi}. \cite{de2015workforce} provide a detailed review of workforce planning models incorporating skills. Other work jointly optimises workforce planning with production \citep{singhal1992noniterative}, capacity planning \citep{song2008successive}, patient and personnel scheduling \citep{koeleman2012optimal}, shift scheduling \citep{kim2015two}, fleet renewal \citep{turan2022joint}, and crowdsourced delivery \citep{cheng2025robust}. These models capture important operational and organisational details, but they are typically limited to two-stage stochastic programming models with recourse --- as opposed to the multi-stage nonstationary settings tackled in this work. Most importantly, they do not derive structural results that concern hiring policies under stochastic turnover and fixed recruitment costs.

The following recent strategic workforce models are especially relevant. \cite{hu2016strategic} formulate an infinite-horizon linear programming model for health workforce planning with training, hiring, promotion, attrition, and supervision constraints. Under common-sense assumptions, they prove optimality of a lookahead policy and show that such policies perform well in deterministic and stochastic demand-growth experiments. Their model uses deterministic retention rates in the core formulation and focuses on lookahead optimality rather than threshold-type hiring. \cite{jaillet2022strategic} propose a robust workforce planning model with uncertain attrition, time-in-grade, hiring, dismissal, and promotion decisions. Their approach operates at a strategic level, provides probabilistic and robustness guarantees, and leads to a tractable conic optimisation formulation. In contrast, our model is a tactical multi-stage stochastic dynamic program with nonstationary binomial turnover and fixed recruitment costs, and our main result is the structural optimality of an $(s,S)$ hiring rule. \cite{luy2024strategic} study hybrid fleets in crowdsourced delivery, where salaried drivers coexist with uncertain crowdsourced-driver supply. They formulate a Markov decision process and use approximate dynamic programming, exploiting convexity along the salaried-driver dimension. Their setting is operationally rich and includes stochastic driver joining and resignation, but the focus is approximate computation for hybrid delivery fleets rather than exact structural characterisation of a fixed-cost hiring policy.

Markov workforce planning provides another important foundation. These models describe personnel movement among grades, departments, or subgroups through recruitment, promotion, internal transitions, and wastage. Classical models study attainability and maintainability of desired personnel structures, while more recent variants incorporate stochastic wastage, uncertain internal flows, and cost-effectiveness objectives \citep{Guerry2012,DeFeyter2016,Komarudin2015}. For example, \cite{DeFeyter2016} develop stochastic Markov workforce models under recruitment control and use scenario-based mixed-integer programming to balance recruitment cost against the desirability of the resulting personnel structure. \cite{Komarudin2015} study trade-offs between attaining a desired structure and maintaining steady promotions when wastage is stochastic. These studies are valuable for modelling personnel structures, but their emphasis is on feasibility, desirability and recruitment vectors, not on identifying the structure of the optimal dynamic hiring decision in a multi-stage nonstationary stochastic setting.

Stochastic programming offers a complementary approach to workforce planning under uncertainty. \cite{Zhu2009} formulate a two-stage stochastic mixed-integer program for multi-category workforce planning under uncertain and fluctuating demand. First-stage decisions determine recruitment and allocation, while second-stage recourse decisions reassign workforce demand among service centres, shifts, or months. Their model accommodates rich operational constraints and uses Benders' decomposition to improve tractability. Our model is less detailed operationally, but it is multi-stage and it provides a structural result.

\textit{Operational staffing and scheduling under uncertainty.}
A further related stream studies operational staffing and scheduling under uncertain availability. In firefighter staffing, annual staffing decisions must account for temporary absences, permanent wastage, limited hiring opportunities, and minimum staffing requirements; the resulting model resembles a newsvendor problem with decision-dependent uncertainty \citep{Fry2006}. In nurse staffing, absenteeism may be endogenous: nurses are more likely to be absent when anticipated workload is high, creating a feedback between staffing decisions and realised availability \citep{Green2013}. Integrated nurse staffing and scheduling models for intensive-care units similarly account for uncertain patient census, acuity, and admission-discharge-transfer activity, and show that dynamic staffing policies can reduce understaffing risk relative to fixed targets \citep{Aydas2020}. These studies share our concern with uncertain workforce availability, but they primarily focus on staffing levels, schedules, or service-level performance rather than on finite-horizon hiring with fixed recruitment costs.

\textit{Fixed adjustment costs.}
Finally, fixed recruitment costs are central to our work. Deterministic dynamic programming models have considered recruitment policies with fixed costs \citep{rao1990dynamic}, and early workforce pooling models recognised adjustment costs in comparing stable and hire-fire policies \citep{Conrath1971}. However, we are not aware of prior work that proves the optimality of an $(s,S)$ hiring policy in a finite-horizon stochastic workforce planning model with fixed recruitment costs and binomial turnover whose distribution depends on the post-hiring workforce level.

\section{Problem setting}\label{sec:problem-setting}


Throughout this paper, $\mathbb{E}[\cdot]$ denotes the expectation operator, with all expectations assumed to be finite. When necessary, a subscript is added to $\mathbb{E}[\cdot]$ to indicate the random variables with respect to which the expectation is taken. For simplicity, we may omit the subscript $t$ in some parameters, reinstating it only when required. We use $\mathbb{R}$ for the set of real numbers, $\mathbb{Z}$ for the set of integers and $\mathbb{N}$ for the set of natural integers including zero. The abbreviations PMF and CDF refer to the probability mass function and cumulative distribution function, respectively.



\subsection{Problem description}

We consider an organisation facing uncertain staff turnover over a discrete finite horizon comprising $T$ periods. The staff base comprises a single employee type; while existing works \citep[see e.g.,][]{jaillet2022strategic} typically consider multiple employee types as well as promotions, these models are not dynamic\footnote{For instance, in \citeauthor{jaillet2022strategic} the recruitment plan is decided once-and-for-all at the beginning of the planning horizon.} and/or do not consider fixed recruitment costs. 
In period $t$, the minimum staffing level that is required to run operations is $W_t$. We assume staff members make independent resignation decisions; in period $t$, a staff member may resign with probability $p_t$, which may vary across periods to reflect nonstationary turnover conditions.

Recruitment of new staff entails a fixed cost $K$ that is independent of the number of people hired, and a variable cost $c$ is paid for recruiting each new staff member. We consider recruitment costs in a general way, including advertisement, screening, and interviewing, but also any possible training cost or even negative economic effects of the adaptation period for the new hires. The organization initially employs $x_0$ staff. In each period, the salary of a staff member is $h$. If the staffing level in a given period is lower than $W_t$, a penalty cost $\pi$ for overtime and/or outsourcing is incurred by the organization for each staff unit short per period. 

At the beginning of each period, the organization decides the number $Q_t$ of people to recruit. The objective of the firm is to minimize the expected total cost, which includes fixed costs, variable costs, salary costs, and penalty costs over a planning horizon of length $T$. Let $x_{t-1}$ denote the staff count at the beginning of period $t$, and let $y_t$ denote the number of staff immediately after recruitment in period $t$, that is $y_t\triangleq x_{t-1}+Q_t$.

Let random variable $\omega^y_t$ denote the number of staff turnovers in period $t$. Since each staff member may resign with probability $p_t$ and resignations happen independently of each other, random variable $\omega^y_t$ follows a binomial distribution with $y_t$ trials and success probability $p_t$. 
%
Staff turnover $\omega^y_{t}$ hence depends upon available staff after recruitment $y_t$.\footnote{The superscript ${}^y$ emphasizes this dependency on $y_t$.}


The staffing level conservation constraint can be expressed as:
\begin{equation}
x_{t}=y_t-\omega^y_t.\label{eq:staff-flow}
\end{equation}

Let $L(y)$ denote the expected total salary cost and penalty cost in period $t$ given available staff number $y$. 
The expression of $L(y)$ is given by Eq. \eqref{eq:Ly}, in which  $(x)^+\triangleq\max(x,0)$. 
\begin{align}
    L(y)\triangleq&\, h\mathbb{E}_{\omega^y}[y-\omega^y] + \pi \mathbb{E}_{\omega^y}[\omega^y+W-y]^+\\
    = &\, h y(1-p_t) + \pi \mathbb{E}_{\omega^y}[\omega^y+W-y]^+.\label{eq:Ly}
\end{align}

For notational simplicity, we hereafter suppress the explicit dependence of the expectation on $\omega^y$, retaining $\omega^y$ only in the random variables themselves. This notation will be reintroduced only when necessary for clarification (e.g., in Section \ref{sec:computation-milp}).

The first term in the right side of $L(y)$ captures the expected salary cost, while the second term captures the expected outsourcing/penalty cost. The dependency of the random staff turnover ($\omega^y$) on the available staff after recruitment ($y$) sets this function apart from the traditional first-order loss function in the inventory management problems.

\subsection{Stochastic dynamic programming modeling}

We next formulate the problem as a stochastic dynamic program.

\vspace{0.5em}
\noindent
\textbf{States.} The state at the beginning of period $t$ is the staffing level $x_{t-1}$. 

\vspace{0.5em}
\noindent
\textbf{Actions.} The action in period $t$ is the recruit-up-to level $y_t$.\footnote{Or equivalently, the recruiting quantity $Q_t=y_t-x_{t-1}$.}

\vspace{0.5em}
\noindent
\textbf{State transition function.} The state transition function is given by Eq. \eqref{eq:staff-flow}.

\vspace{0.5em}
\noindent
\textbf{Immediate cost.} The immediate cost for period $t$ is $\delta(y-x)+L(y)$, where $\delta(\cdot)$ is defined as
\begin{equation}
\delta(z)=
\begin{cases}
K+cz\quad & z>0\\
0 & z=0.
\end{cases}\label{eq:hiring-cost}
\end{equation}

\noindent
\textbf{Bellman's equation.} Define $C_{t}(x)$ as the minimum expected total cost incurred over periods $t, t+1,\dots, T$, when the initial staffing level is $x_{t-1}$. Bellman's equation for this problem is then
\begin{equation}
C_{t}(x)=\min\limits_{y\geq x }\Big\{\delta(y-x)+L(y)+\mathbb{E}[C_{t+1}(y-\omega^y)]\Big\}\label{eqn:c(x)}
\end{equation}
with boundary condition $C_{T+1}(x)\triangleq 0$. 

As in \cite{Scarf1960}, we introduce the function 
\begin{equation}
    G_t(y)\triangleq cy+L(y)+\mathbb{E}[C_{t+1}(y-\omega^y)]\label{eq:Gy}
\end{equation}
which will be used to develop our analysis. Note that $C_t(x)=G_t(x)-cx$ when not hiring.






\section{Mathematical properties}\label{sec:math-property}

In this section, we first establish the discrete convexity of $L(y)$. We then introduce the concept of $K$-convexity and characterize its properties in our problem featuring Binomial decision-dependent propagations. Building on these results, we formally prove that $G(y)$ is $K$-convex and establish the optimality of the $(s, S)$ policy. To maintain clarity in the proofs, we suppress the time subscript on the turnover probability and denote it by 
$p$ throughout this section. However, all established properties hold for each period $t$ with its specific parameter 
$p_t$ , as the convexity arguments rely on the binomial structure rather than the specific value of the probability.





\subsection{Discrete convexity of $L(y)$}

For ease of expression, we remove the term $h y(1-p)$ from $L(y)$ in Eq. \eqref{eq:Ly}, which is linear in $y$; and the constant $\pi$, which is assumed positive. 
 Let $g:\mathbb{N}\rightarrow \mathbb{R}$ be resulting function:
\begin{equation}\label{eq:gDefinition}
g(y)\triangleq\mathbb{E}[\omega^y+W-y]^+.
\end{equation}

We next introduce the concept of discrete convexity and prove that $g(y)$ and $L(y)$ are discrete convex.


\begin{definition}
A function $f:\mathbb{Z}\rightarrow \mathbb{R}$ is discrete convex if
\begin{equation}\label{eq:conv}
    f(x-1)+f(x+1)\geq 2f(x)\qquad \forall x\in\mathbb{Z}.\nonumber
\end{equation}
\end{definition}

The following lemma can be easily verified.
\begin{lemma}\label{lemma:discrete-connvex}
A non-negative weighted sum of discrete convex functions is discrete convex.
\end{lemma}




\begin{lemma}
$g(y)$ and $L(y)$ are discrete convex.\label{lemma:gLconvex}
\end{lemma}



\begin{proof}

Clearly, $L(y)$ is discrete convex if $g(y)$ is discrete convex. We now prove the discrete convexity of $g(y)$. 

Recall that there is a key difference between this function and the first order loss function in inventory control, which is known to be convex. In our case, the random variable $\omega^y$ varies with $y$; hence, convexity is not evident in general. 
Let $\beta$ denote a Bernoulli random variable
with success probability $p$, independent of $\omega^y$. Then, since $\omega^{y+1}=\omega^y+\beta$, we obtain: \begin{align}
g(y+1)
&= \mathbb{E}\bigl[\omega^{y}+\beta+W-y-1\bigr]^+ \nonumber\\
&= p\,\mathbb{E}\bigl[\omega^{y}+W-y\bigr]^+
  + (1-p)\,\mathbb{E}\bigl[\omega^{y}+W-y-1\bigr]^+. \label{eq:gyplus1}
\end{align}

For fixed \(y\), define
\[
H(m)\triangleq \mathbb{E}\bigl[\omega^y+W-y-m\bigr]^+.
\]
Then
\[
g(y)=H(0)
\]
and, by \eqref{eq:gyplus1},
\[
g(y+1)=pH(0)+(1-p)H(1).
\]
Therefore,
\[
g(y+1)-g(y)
=(1-p)\bigl[H(1)-H(0)\bigr].
\]

Applying the same argument once more, let \(\beta_1\) and \(\beta_2\) be independent \(\operatorname{Bernoulli}(p)\) random variables, independent of \(\omega^y\), so that
\[
\omega^{y+2}=\omega^y+\beta_1+\beta_2.
\]
It follows that
\[
g(y+2)
=p^2H(0)+2p(1-p)H(1)+(1-p)^2H(2).
\]
Consequently,
\begin{align}
&[g(y+2)-g(y+1)]-[g(y+1)-g(y)] \notag\\
&\qquad
=(1-p)^2\bigl[H(0)+H(2)-2H(1)\bigr].
\label{eq:second-diff}
\end{align}

Hence, to establish the discrete convexity of $g(y)$, it suffices to show that
\[
H(0)+H(2)-2H(1)\ge 0.
\]


Recall that $H(m)$ is analogous to the inventory first-order loss function with an order-up-to level of $y+m-W$, which is well known to be convex in the continuous setting. Indeed, $[\omega^y + W- y - m]^+ = \max\{0, \,\omega^y + W - y - m \}$ is a discrete convex function of $m$. Since $H(m)$ is a non-negative weighted sum of such discrete convex functions, it is itself discretely convex by Lemma~\ref{lemma:discrete-connvex}. By the definition of discrete convexity, it then follows that $g(y)$ is discretely convex in $y$. 
\end{proof}

A number of existing works have suggested in the past that convexity of $g(y)$ may be established by leveraging total positivity \citep{karlin1968total}, or super-modularity \citep{ahn2005staffing}. However, none of these works provide a complete proof. Our proof, which leverages a Bernoulli trial decomposition of $g(y)$ and a transformation involving function $H(\cdot)$, is intuitive and direct.

\subsection{$K$-convexity of $G(y)$ and optimality of $(s, S)$ policies}

We modify the definition of $K$ convexity introduced by \cite{Scarf1960} on natural numbers as follows.

\begin{definition}
A function $f:\mathbb{N}\rightarrow \mathbb{R}$ is $K$-convex if for all $x$, $a$ and $b$,
\begin{equation}
    b\left(K+f(x+a+b)-f(x)\right)\geq (a+b)\left(f(x+b)-f(x)\right).
\end{equation}\label{def:Kconvex}
\end{definition}

For $K$-convex functions, there are some properties given by the following lemma, which are crucial for the optimality of $(s, S)$ policy. 

\begin{lemma}\label{lemma:K_properties}
\begin{enumerate}[label=\normalfont{({\alph*})} ]

\item\label{item-Kconvexity:a} A discrete convex function $f$ is also 0-convex and hence also $K$-convex for all $K\geq 0$.

\item\label{item-Kconvexity:b} If $f$ is $K$-convex and $g$ is $L$-convex, then for any scalars $\alpha, \beta \ge 0$, the function $\alpha f + \beta g$ is $(\alpha K + \beta L)$-convex.

\item\label{item-Kconvexity:c} Let $f$ be $K$-convex and $\omega^y \sim \mathrm{Binomial}(y, p)$. If $\mathbb{E}[|f(y - \omega^y)|] < \infty$, then the function $\mathbb{E}[f(y - \omega^y)]$ is $K$-convex.


\item\label{item-Kconvexity:d} If $f$ is a $K$-convex function and $f(x)\rightarrow \infty$ for $x\rightarrow \infty$, there exist scalars $s$ and $S$, defined as follows
\begin{align}
    S&=\min\{\arg\min_{x\in\mathbb{N}} f(x)\},\notag\\
    s&=\min\{x\in\mathbb{N}\mid  f(x)\leq f(S)+K, x\leq S\}\notag,
\end{align}
such that
\begin{enumerate}[label=\normalfont{(\roman*)} ]
\item $f(S)\leq f(x)$ for all  $x$;
\item $f(x)$ is strictly decreasing on the integers in [0, s);
\item $f(x)\leq f(z)+K$ for all $x$, $z$ with $s\leq x\leq z$.
\end{enumerate}
\end{enumerate}\label{lemma:Kconvex}
\end{lemma}

\begin{proof}
The proofs of parts \ref{item-Kconvexity:a}--\ref{item-Kconvexity:b} and \ref{item-Kconvexity:d} are very similar to those in \citet{bertsekas2005dynamic}. However, the proof of part \ref{item-Kconvexity:c} is less straightforward, as the expectation operator explicitly involves the decision variable $y$.

Fix \(y,a,b\in\mathbb{N}\). Define
\[
Y=y-\omega_y,\qquad
A=a-\omega_a,\qquad
B=b-\omega_b,
\]
where \(\omega_y,\omega_a,\omega_b\) are independent binomial random variables with parameters \((y,p)\), \((a,p)\), and \((b,p)\), respectively. Hence,
\[
Y\sim\operatorname{Binomial}(y,1-p),\qquad
A\sim\operatorname{Binomial}(a,1-p),\qquad
B\sim\operatorname{Binomial}(b,1-p).
\]

For any realizations \(x,r,u\) of \(Y,A,B\), respectively, the \(K\)-convexity of \(f\), as defined in (10), gives
\[
u\left(K+f(x+r+u)-f(x)\right)
\ge
(r+u)\left(f(x+u)-f(x)\right).
\]

Let
\[
N=A+B.
\]
By the additive property of independent binomial random variables,
\[
N\sim\operatorname{Binomial}(a+b,1-p).
\]
Conditioning on \(N=n\), we obtain
\[
\mathbb{E}[B\mid N=n]
\left(K+f(x+n)-f(x)\right)
\ge
n\,
\mathbb{E}\left[
f(x+B)-f(x)\mid N=n
\right].
\]

Given \(N=n\), \(B\) follows a hypergeometric distribution, and therefore
\[
\mathbb{E}[B\mid N=n]
=
\frac{b}{a+b}n.
\]
Thus, for \(n>0\),
\[
b\left(K+f(x+n)-f(x)\right)
\ge
(a+b)
\mathbb{E}\left[
f(x+B)-f(x)\mid N=n
\right].
\]
The same inequality holds trivially when \(n=0\). Taking expectations with respect to \(N\) gives
\[
b\left(
K+\mathbb{E}[f(x+N)]-f(x)
\right)
\ge
(a+b)\left(
\mathbb{E}[f(x+B)]-f(x)
\right).
\]

Since the above inequality holds for every realization \(x\) of \(Y\), taking expectations with respect to \(Y\) yields
\[
b\left(
K+\mathbb{E}[f(Y+A+B)]-\mathbb{E}[f(Y)]
\right)
\ge
(a+b)\left(
\mathbb{E}[f(Y+B)]-\mathbb{E}[f(Y)]
\right).
\]

Finally, using again the additive property of independent binomial random variables,
\[
Y+A+B\triangleq y+a+b-\omega_{y+a+b},
\]
and
\[
Y+B\triangleq y+b-\omega_{y+b}.
\]
Therefore,
\[
b\left(
K+
\mathbb{E}[f(y+a+b-\omega_{y+a+b})]
-
\mathbb{E}[f(y-\omega_y)]
\right)
\]
\[
\ge
(a+b)\left(
\mathbb{E}[f(y+b-\omega_{y+b})]
-
\mathbb{E}[f(y-\omega_y)]
\right),
\]
which is exactly the \(K\)-convexity condition in (10) for the function
\[
\mathbb{E}[f(y-\omega_y)].
\]
Hence, \(\mathbb{E}[f(y-\omega_y)]\) is \(K\)-convex.

\end{proof}

\begin{theorem}
$G_t(x)$ and $C_t(x)$ are \(K\)-convex and the optimal hiring policy for each period is $(s, S)$, i.e.,
\begin{equation}
    Q_t^\ast=\begin{cases}
    S_t-x_{t-1}\quad & x_{t-1}< s_t\\
    0 & x_{t-1}\ge s_t,
    \end{cases}\label{eq:sS}
\end{equation}
where $(s,S)$ are defined as in Lemma \ref{lemma:K_properties} \ref{item-Kconvexity:d}:
\begin{align}
    S&=\min\{\arg\min_{x\in\mathbb{N}} G_t(x)\},\\
    s&=\min\{x\in\mathbb{N}\mid  G_t(x)\leq G_t(S)+K, x\leq S\}\notag.
\end{align}
\label{theorem-K-convex}
\end{theorem}


In our problem, as the decision variables are integers, Theorem \ref{theorem-K-convex} follows immediately from \cite{Scarf1960} and Lemma \ref{lemma:K_properties}. For completeness, the proof is provided in Appendix~\ref{app:theorem-optimality}.

\section{A mixed-integer linear programming heuristic}\label{sec:computation-milp}

Theorem \ref{theorem-K-convex} provides a characterization for computing the optimal $(s_t,S_t)$ parameters at each period $t$, which requires finding the $\arg\min$ $S$ of function $G_t$ and the minimum $s$ such that $s\leq S$ and $G_t(s)\leq G_t(S)+K$. However, $G_t$ lacks a closed-form expression and is, in fact, as difficult to optimize as the original dynamic programming problem itself. This is a classical challenge in inventory theory, and the literature on $(s_t,S_t)$ heuristics spans more than six decades --- see, for instance, the recent survey by \cite{perera2023survey}.


In this section, we present an efficient approach to approximate $G_t(\cdot)$ and determine the parameters $(s_t, S_t)$ for each period. This is achieved by leveraging methodologies from classical inventory control. Specifically, the approach in inventory rests on two key approximations: (1) adopting the static-dynamic uncertainty strategy introduced by \cite{bookbinder1988strategies}, where review periods (corresponding to recruitment periods in our context) are predetermined at the beginning of the planning horizon, whereas quantities are decided dynamically when orders/recruitment are placed; and (2) employing piecewise linear approximations to tractably handle model non-linearities.

The adaptation of such approach to workforce is not direct. Indeed, regarding approximation (1), a central concept is the aggregated demand over a cycle --- i.e., the span between consecutive ordering or hiring decisions. While inventory models simply sum demands across periods, workforce turnover exhibits inter-period dependencies that render it non-additive. Nevertheless, Proposition \ref{prop:xt} establishes a closed-form expression for the distribution of this aggregated turnover. Regarding approximation (2), the discrete nature of workforce planning plays to our advantage. Consequently, the non-linearities can be effectively captured using a finite set of linear constraints. Section \ref{sec:linearization} details the precise formulation and explains the approximation mechanism.


\subsection{Stochastic programming formulation}

We first reformulate our problem as a multistage stochastic program (MSP). At each stage $t$, $y_t$ is the primary decision variable and represents the workforce level after recruitment. The recruitment indicator $z_t$ is an auxiliary binary variable associated with the recruitment decision. Another auxiliary variable $U_t$ is the number of outsourced staff in period $t$. 
Once $y_t$ is determined and the uncertainty $\omega_t^y$ is realized, $x_t$ and $U_t$ are determined by constraints \eqref{con:msp-xy} and  \eqref{con:msp-U}, respectively.

While this problem is conventionally solved via dynamic programming, the approach suffers from the curse of dimensionality as the time horizon $T$ grows large. We introduce the MSP formulation for two primary reasons: first, it elucidates the nested structure of the sequential decision process, and second, it serves as the foundation for the MILP formulation presented in Section \ref{sec:MILP}.

The MSP is formulated as follows, where $M$ is a large enough number \footnote{In our numerical experiments, $M$ is set to the product of the planning horizon length $T$ and the total minimum required workforce $\sum_{t=1}^T W_t$.}:
\begin{align}
   \min_{y_1}\quad
   &c(y_1-x_0)+Kz_1
   +\mathbb{E}_{\omega_1^y}
   \Bigg[
      hx_1+\pi U_1
      +\min_{y_2}
      \Bigg\{
         c(y_2-x_1)+Kz_2\nonumber\\
         &
         +\mathbb{E}_{\omega_2^y}
         \Bigg[
            hx_2+\pi U_2
            +\min_{y_3}
            \Bigg\{
               \cdots
               +\min_{y_T}
               \Big\{
                  c(y_T-x_{T-1})+Kz_T
                  +\mathbb{E}_{\omega_T^y}
                  \left[
                     hx_T+\pi U_T
                  \right]
               \Big\}
            \Bigg\}
         \Bigg]
      \Bigg\}
   \Bigg]
   \label{obj:msp}\\
   \text{s.t.}\quad
   &x_t=y_t-\omega_t^y,
   \qquad t=1,2,\ldots,T,
   \label{con:msp-xy}\\
   &y_t-x_{t-1}\geq 0,
   \qquad t=1,2,\ldots,T,
   \label{con:msp-yxt1}\\
   &y_t-x_{t-1}\leq z_tM,
   \qquad t=1,2,\ldots,T,
   \label{con:msp-zt}\\
   &U_t=\max\left\{W+\omega_t^y-y_t,0\right\},
   \qquad t=1,2,\ldots,T,
   \label{con:msp-U}\\
   &x_t\ge 0,\quad y_t\in\mathbb{N},\quad U_t\geq0,\quad
   z_t\in\{0,1\},\qquad t=1,2,\ldots,T.
   \label{con:msp-domain}
\end{align}


The objective function \eqref{obj:msp} minimizes the expected total cost over the planning horizon. In each period $t$, $c(y_t-x_{t-1})$ represents the recruitment cost, $Kz_t$ represents the fixed recruitment cost, $hx_t$ represents the cost associated with the workforce salaries, and $\pi U_t$ represents the outsourcing cost.

Constraints \eqref{con:msp-xy} determine $x_t$ after the realization of the random variable $\omega_t^{y}\sim  \mathrm{Binomial}(y_t, p_t).$  Constraints \eqref{con:msp-yxt1} ensure that the workforce level after recruitment cannot be lower than the workforce level before recruitment. In addition, Constraints 
\eqref{con:msp-zt} link the binary recruitment indicator $z_t$ with the recruitment amount. In particular, if recruitment occurs, i.e., $y_t>x_{t-1}$, then $z_t=1$. If no recruitment occurs, i.e., $y_t=x_{t-1}$, then $z_t=0$ is optimal when $K>0$. Constraint \eqref{con:msp-U} determines the number of outsourced staff required to compensate for the workforce shortage.

The nested structure of \eqref{obj:msp} reflects the sequential decision-making process. Specifically, $y_t$ is determined before the realization of the current-period uncertainty $\omega_t^y$, whereas $y_{t+1}$ can depend on the realization of $\omega_t^y$. Indeed, if no hiring is done, i.e. $z_t=0,$ then constraints (\ref{con:msp-yxt1}) and (\ref{con:msp-zt}) ensure $y_t=x_{t-1},$ hence being affected by the random variable $\omega_{t-1}^y,$ binomial with $y_{t-1}$ trials. The decision at stage $t$ satisfies the information structure
\(y_t = y_t(\omega_1^y,\ldots,\omega_{t-1}^y),\)
where $y_1$ is determined based only on the initial information. Thus, the nested expectation and minimization in \eqref{obj:msp} implicitly enforce non-anticipativity constraints: decisions at stage $t$ cannot depend on uncertainty realizations that have not yet been observed.




\subsection{``Static-dynamic” MILP formulation}\label{sec:MILP}


In order to formulate an MILP model, the first step is to approximate it by the so-called ``static-dynamic” control strategy. To implement this strategy, we introduce a new binary decision variable, $P_{jt}$, which equals one if and only if the most recent staff hiring before period $t$ occurred in period $j$. 
The decision variables $z_t$ and $P_{jt}$ are fixed in a here-and-now fashion at the beginning of the planning horizon. In hiring periods, i.e. $z_j=1,$ the hire-up-to-level $y_j$ are also fixed and decided at the beginning of the planning horizon. When not hiring, $y_j$ will depend on the randomness of turnover.  We refer to the interval from period $j$ to period $t$ as a \textit{hiring cycle}\footnote{The ``static-dynamic" strategy further assumes that cases where the initial staffing level exceeds the hire-up-to target in any hiring cycle are rare and can be safely ignored.}. Proposition \ref{prop:xt} describes the distribution of turnover during the cycle.

\begin{proposition}\label{prop:xt}
    Let $x_t = Y - \omega_t^Y$ be a random variable where $Y = y_t$ captures the randomness that depends on all previous periods $k = 1, 2, \ldots, t$ through the relationships governed by Constraints~\eqref{con:msp-yxt1} and \eqref{con:msp-zt}. If $P_{jt} = 1,$ then $x_t \sim \mathrm{Binomial}(y_j, q_{jt}),$ with
\begin{equation}\label{def:qjt}
q_{jt} = \prod_{k=j}^t (1-p_k).
\end{equation}
\end{proposition}

\begin{proof}
When $P_{jt}=1$, the random variable $x_t$ depends solely on the periods within its hiring cycle, i.e., $k=j,\ldots,t$, where each $y_k$ depends on $y_{k-1}$ for all $k=j+1,\ldots,t$. However, $y_j$ is decoupled from $y_{j-1}$ because the hiring process is activated, fixing its value at the beginning of the horizon. Rather than modelling the distribution of $x_t$ as a composition of binomial distributions, we compute the probability that a worker present in the workforce at period $j$ remains employed at the end of period $t$. This can be interpreted as a survival probability, i.e., the probability of not departing the firm during periods $k=j, j+1, \ldots, t$. Assuming independent and identically distributed (i.i.d.) retention probabilities $1-p_k$ across workers at each period $k$, the overall survival probability for each of the $y_j$ workers is given by $q_{jt}$ in Equation~\eqref{def:qjt}. Given that individual departure decisions are mutually independent, the desired result follows directly.
\end{proof}

The expectation formula for $x_t$ in a hiring cycle from period $j$ to $t$ can be rewritten as  
\[
\mathbb{E}[x_t] = y_j \prod_{k=j}^t (1-p_k).
\]

The resulting MILP model is formulated as follows. First, the original objective function \eqref{obj:msp} is converted into \eqref{obj:piecewise} by simplifying the nested cost-to-go expressions into additive sums (note that $\mathbb{E}[x_{0}]=x_0$). Next, the random variables $x_t$ and $U_t$ are replaced with their respective expectations, $\mathbb{E}[x_t]$ and $\mathbb{E}[U_t]$. Let $\tilde{x}_t$ and $\tilde{U}_t$ denote the decision variables representing those expectation in the MILP formulation. Finally, the integer requirements on $y_t$ are relaxed to non-negativity for computational efficiency.

In addition, Constraints \eqref{con:xtyj1}--\eqref{con:xtyj2} govern the calculation of $\tilde{x}_t,$ i.e. $\mathbb{E}[x_t]$, within a hiring cycle spanning from period $j$ to $t$, represented by the auxiliary binary variable $P_{jt}$: if $P_{jt}=1$, $\tilde{x}_t$  evaluates to the expectation formula  associated with this specific cycle. Meanwhile, Constraint \eqref{con:sumPjt} ensures that exactly one most recent hiring event occurs prior to period $t$. Combined with Constraint \eqref{con:Pjtzj}, the two conditions define the relationship between $P_{jt}$ and $z_j$:
\begin{itemize}
    \item If $z_j=1$ and $\sum_{k=j+1}^{t}z_k=0$, then $P_{jt}=1$, indicating that the most recent hiring cycle indeed originated in period $j$;
    \item If $z_j=1$ and $\sum_{k=j+1}^{t}z_k\ge 1$, then $P_{jt}=0$, as the most recent hiring cycle preceding period $t$ must have started after period $j$;
\end{itemize}

Finally, the only non-linear constraint in the model is \eqref{con:EUt}, which ensures that $\tilde{U}_t$ equals the expected number of outsourced staff in period $t$. In Section \ref{sec:linearization}, we provide an exact linearization of this constraint. To this end, we first introduce the following Mixed-Integer Non-linear Programming (MINLP) model, in which only  constraints \eqref{con:EUt} are non-linear:

\begin{align}
    &\min&& \sum_{t=1}^T \left(c(y_t-\tilde{x}_{t-1}) + K z_t+h\tilde{x}_t+\pi \tilde{U}_t \right) \label{obj:piecewise}\\
    & \quad \text{s.t.}&& \tilde{x}_t\ge y_j\prod_{k=j}^t (1-p_k)-(1-P_{jt})M, && j=1,2,\dots, t; \, \forall t \label{con:xtyj1}\\
    &&& \tilde{x}_t\le y_j\prod_{k=j}^t (1-p_k)+(1-P_{jt})M, && j=1,2,\dots, t; \,\forall t \label{con:xtyj2}\\
    &&& y_t-\tilde{x}_{t-1}\ge 0, && \forall t \\ 
    &&& y_t - \tilde{x}_{t-1} \le z_t M, && \forall t \label{con:ytxt-1-peice}\\
    &&& \sum_{j=1}^t P_{jt}=1, && \forall t \label{con:sumPjt} \\
    &&& P_{jt}\ge z_j-\sum_{k=j+1}^{t}z_k, && j=1,2,\dots, t; \, \forall t \label{con:Pjtzj} \\
    &&& \tilde{U}_t = \mathbb{E}_Y\left[W+\omega_t^Y-Y_t\right]^+, && \forall t \label{con:EUt}\\
    &&& P_{jt}\in\{0,1\}, && j=1,2,\dots, t; \, \forall t\\
    &&& \tilde{x}_t\ge 0, \, \tilde{U}_t\ge 0, \, y_t\ge 0, \, z_t \in\{0, 1\}, && \forall t \label{con:non-negative-piecewise}.
\end{align}



\subsection{Linearization of the loss function}\label{sec:linearization}

While piecewise linear approximation methods have been employed in prior literature \citep[e.g.,][]{Rossi2015} for inventory loss functions, the linearization of $\mathbb{E}[U_t]$ in our formulation presents a distinct challenge because it depends directly on the decision variable $Y_t-\omega^Y_t$ where $Y$ is a random variable that carries on all the consequences of decisions and uncertainty realizations in periods before $t$. 

From Proposition \ref{prop:xt}, if $P_{jt}=1,$ then $Y_t-\omega^Y_t \sim \mathrm{Binomial}(y_j,q_{jt})$. Consequently, $\mathbb{E}[U_t]$ can be expressed as $g(y_j)$, where $g$ is defined as in Equation \eqref{eq:gDefinition} with turnover probability $p=1-q_{jt}$. Recall that $g(y)$ is discrete convex, as established in Lemma \ref{lemma:gLconvex}. To emphasize the dependence on the parameter $1-q_{jt}$, we denote this function by $g_{jt}(y)$. Specifically, $g_{jt}(y) = \mathbb{E} \left[ W + \omega^y - y \right]^+$, where $\omega^y \sim \mathrm{Binomial}(y, 1-q_{jt})$.

By applying this definition, constraint \eqref{con:EUt} can be replaced by the implication
\[P_{jt}=1 \implies \tilde{U}_t = g_{jt}(y_j), \quad \forall j \in \{1, \dots, t\}, \, \forall t.
\]

Although a standard big-$M$ formulation reformulates this implication into a linear constraint, $g_{jt}$ itself remains non-linear. To address this, we linearize $g_{jt}$ by leveraging its forward-difference properties (see Lemma \ref{prop:forward-diff} below). The precise linearization is formalized in Proposition \ref{prop:lin}.


\begin{lemma}\label{prop:forward-diff}
The forward difference of $g_{jt}(y)$ in a hiring cycle from period $j$ to period $t$ is given by
\begin{equation}
    g_{jt}(y+1) - g_{jt}(y) = -q_{jt} \left( 1 - F^y(y-W) \right), \label{eq:forward-diff2}
\end{equation}
where $F^y(\cdot)$ is the cumulative distribution function of a $\mathrm{Binomial}(y, 1-q_{jt})$ random variable.
\end{lemma}

The proof is provided in Appendix \ref{sec:lemmaForwardDiff}.

\begin{proposition}\label{prop:lin}
Problem \eqref{obj:piecewise}--\eqref{con:non-negative-piecewise} is equivalent to the MILP model for integer $y_t$ in which the non-linear constraints \eqref{con:EUt} are replaced by
\begin{equation}
    \tilde{U}_t \geq \alpha_{kjt} y_j + \beta_{kjt} - (1-P_{jt})M, \qquad k=1, 2,\dots, \overline{W}, \quad j=1,\dots,t, \quad\forall t,
    \label{con:piecewise}
\end{equation}
where $\overline{W}$ is a sufficiently large upper bound of $y_j$, and
\begin{align*}
\alpha_{kjt} &= g_{jt}(k+1) - g_{jt}(k),\\
\beta_{kjt} &= g_{jt}(k) - k \big( g_{jt}(k+1) - g_{jt}(k) \big).
\end{align*}
Equivalently, these coefficients can be written explicitly as
\begin{align*}
\alpha_{kjt} &= - q_{jt} \left( 1 - F^k(k-W) \right), \\
\beta_{kjt} &= g_{jt}(k) + k q_{jt} \left( 1 - F^k(k-W) \right).
\end{align*}
\end{proposition}

\begin{proof}
When $P_{jt}=1$, constraint \eqref{con:EUt} reduces to $\tilde{U}_t = g_{jt}(y_j)$. Since $g_{jt}$ is discrete convex (Lemma \ref{lemma:gLconvex}), it can be linearized using its subgradients: for any $k \in \mathbb{N}$, we have
\[
g_{jt}(x) \geq \alpha x + \beta, \qquad \forall x \in \mathbb{N},
\]
where $\alpha$ and $\beta$ define a line passing through the points $(k, g_{jt}(k))$ and $(k+1, g_{jt}(k+1))$. Constructing the line through these two points yields the first set of expressions for $\alpha_{kjt}$ and $\beta_{kjt}$ in Proposition \ref{prop:lin}.

The second set of expressions follows directly by applying Lemma \ref{prop:forward-diff}.

To show that a finite number of sub-gradients suffices, for \(y<W\), every realised workforce level after turnover is below $W$, Hence,
\[ g_{jt}(y) = W-E[y-\omega^y] = W-yq_{jt},
\]
where \(q_{jt}\) is the survival probability defined in \eqref{def:qjt}. Consequently, all sub-gradients for $k = 1, \dots, W-1$ are identical. For the maximum index $\overline{W}$, we can select any upper bound on $y_j$. For instance, setting $\overline{W} = W \times T$ corresponds to the conservative bound where the entire workforce leaves in each period.

Finally, a standard big-$M$ formulation ensures that these linear constraints are enforced only when $P_{jt}=1$.
\end{proof}

Proposition \ref{prop:lin} provides an exact equivalence to Problems (\ref{obj:piecewise})\ldots(\ref{con:non-negative-piecewise}). However, this approach utilizes every sub-gradient of $g(y)$. To address this, one possible strategy is to partition the domain of $g(y)$ into $N$ segments and use only a subset of the sub-gradients. The procedure for dividing the domain is described below.

Since $\lim_{y\rightarrow +\infty} F^y(y-W) = 1$ for $y \geq W$, and $F^y(y-W) = 0$ for $y < W$, we employ a heuristic procedure to partition the domain of $g(y)$ into $N$ segments. The end point of each segment will provide a sub-gradient line based on the variation of $F^y(y-W)$:  
\begin{itemize}
    \item The $1_\text{st}$ sub-gradient is constructed at the point with y-coordinate $W-1$, with slope $-q_{jt}$.  

    \item The $N$-th sub-gradient is constructed at $y$-coordinate $y_N$ with a zero slope, chosen such that $F^y(y_N - W) \approx 1$. In the numerical experiments, we set $F^y(y_N - W) = 0.9999$. 

    \item For $2 \leq i \leq N-1$, the $i_\text{th}$ sub-gradient is constructed at the tangent point with y-coordinate $y_i$, where $y_i$ is the smallest integer satisfying:  
    \[
    F^y(y_i-W) - F^y(y_{i-1}-W) > \frac{1}{N-1}.
    \]

    If no such $y_i$ exists, we set $y_i = y_{i-1}$.  
\end{itemize}

Note that the slopes and intercepts for each segment can be precomputed prior to solving the MILP model to accelerate the overall computation.

The vertical coordinate of each considered point is obtained by evaluating $g(y)$ at the corresponding horizontal coordinate. Figure~\ref{fig: gy-piecewise} illustrates the piecewise linear approximation of $g(y)$ with $W=50$, $p=0.3$, and three sub-gradient segments.

\begin{figure}[!ht]
\centering
\subfigure[$g(y)$ before piecewise approximation]{
\pgfplotsset{height=6.8cm} 
\begin{tikzpicture}
\begin{axis}[xlabel=$y$,ylabel=$g(y)$,xmin=39,xmax=100,ymin=0,ymax=23,clip mode=individual,ytick=\empty,tick label style={font=\small},
y label style={at={(0,0.5)}},
x label style={at={(0.5,-0.1)}}]

\pgfplotstableread{gy.txt}\mydata;
\addplot [only marks, mark options = {red}, mark size=1pt]
table [x expr=\thisrowno{0}, y expr=\thisrowno{1}] {\mydata};

\end{axis}
\end{tikzpicture}
}
\subfigure[$g(y)$ after piecewise approximation]{
\pgfplotsset{height=6.8cm} 
\begin{tikzpicture}
\begin{axis}[xlabel=$y$,xmin=39,xmax=100,ymin=0,ymax=23,clip mode=individual,ytick=\empty,tick label style={font=\small},
y label style={at={(0,0.5)}},
x label style={at={(0.5,-0.1)}}]

\pgfplotstableread{gy.txt}\mydata;
\addplot [only marks, mark options = {red}, mark size=1pt]
table [x expr=\thisrowno{0}, y expr=\thisrowno{1}] {\mydata};

\addplot[color=blue, line width=0.5pt] coordinates{ (40, 22) (67.24, 2.93)};

\addplot[color=blue, line width=0.5pt] coordinates{ (67.24, 2.93) (76.10, 0.00001)};
\addplot[color=blue, line width=0.5pt] coordinates{ (76.10, 0.00001) (100, 0) };

\end{axis}
\end{tikzpicture}
}
\caption{Piecewise approximation for $g(y)$ with 3 sub-gradient segments }\label{fig: gy-piecewise}
\end{figure}




\subsection{Computation of $s$ and $S$}\label{sec:ss_computation}

 Finally, the values of $s$ and $S$ can be obtained by following a strategy akin to that discussed by \cite{xiang2018computing}. The detailed procedure is presented in Algorithm \ref{algorithm1}. The main idea is as follows: for each period, construct the MILP-S model (line \ref{alg:line_1}): by setting $z_1 = 0$, the objective value serves as an approximation of $G_t(x_0) - cx_0$. Moreover, $x_0$ is treated as a decision variable and the optimal value of $x_0$ provides an approximation of $S$ in period $t$ (line \ref{alg:line_2}), while the objective value approximates $G_t(S) - cS$. Finally, determine the approximate integer value of $s$ via a binary search (line \ref{alg:line_3}) over the interval $[0, S]$, such that $G_t(s) \approx G_t(S) + K$.
Since $S$ must be an integer in our discrete problem, we round the value of $S$ obtained from the MILP-S solution. 

\begin{algorithm}[!ht]
\caption{The binary search algorithm to compute $s$ and $S$}\label{algorithm1}
 \KwData{parameter values of the problem, stepsize. }
 \KwResult{approximate values $s$ and $S$ in each period.}

\For {$t\leftarrow 1$ \KwTo $T$}
{          
build the MILP-S model from current period $t$ to $T$\;\label{alg:line_1}
solve MILP-S and get the approximate value of $S$, objective value $\mathrm{obj}_S$ and $G_t(S)=\mathrm{obj}_S+cS$\; \label{alg:line_2}
low $\leftarrow 0$; high $\leftarrow S$\;
$S\leftarrow \text{round}(S)$; mid $\leftarrow$ high/2\;
\While{low $<$ high}{\label{alg:line_3}
mid $\leftarrow$ \text{round}((high+low)/2)\;
run MILP-S with $x_0\leftarrow$ mid, get objective value $\mathrm{obj}_{\mathrm{mid}}$ and obtain $G_t(\text{mid})=\mathrm{obj}_{\mathrm{mid}}+c*\mathrm{mid}$\;
\uIf{$G_t(\mathrm{mid})< G_t(S)+K$}{
high $\leftarrow$ mid - stepsize\;
}
\uElseIf{$G_t(\mathrm{mid})> G_t(S)+K$ }
{
low $\leftarrow$ mid + stepsize\;
}
\Else{
low $\leftarrow$ high;
}
}
$s\leftarrow$ mid\;
}
\end{algorithm}





\section{Numerical study}\label{sec:numerical-tests}

In this section, we present our computational study. All numerical experiments are conducted on a MacBook equipped with an Apple M1 Pro CPU and 16~GB of RAM. The MILP models are solved using Gurobi~13.0.2. Our baseline is an exact stochastic dynamic programming algorithm that solves the problem to optimality.
Both the dynamic programming algorithm and the MILP formulation are implemented in C++.

We evaluate the proposed piecewise approximation method on a 12-period test bed. In the numerical tests, we consider four turnover rate ($p$) patterns: stationary, increasing, decreasing, and seasonal. The minimum staffing requirement ($W$) follows ten patterns: one stationary pattern (STA), two life-cycle patterns (LCY1 and LCY2), two sinusoidal patterns (SIN1 and SIN2), one random pattern (RAND), and four empirical patterns (EMP1, EMP2, EMP3, and EMP4) --- these common demand patterns in the literature \citep[e.g.,][]{Rossi2015}. The fixed hiring cost ($K$), salary cost ($h$) and unit outsourcing (penalty) cost ($\pi$) are each considered at three levels, while the initial number of staff is fixed at zero and unit variable cost ($c$) is fixed at 100.
The parameter variations used in the test bed are summarized in Table \ref{tab:tests-parameters}. In total, we consider 1080 instances. The piecewise approximation employs two sub-gradient segments: the first is positioned at $W-1$ with a slope of $-q_{jt}$, and the second coincides with the horizontal axis. We utilize 2 segments because empirical results show they are sufficient to achieve excellent performance on almost all the cases.

\begin{table}[!ht]
    \centering
     \caption{Parameter variations in the tests}
    \label{tab:tests-parameters}
    \begin{tabular}{ll}
    \toprule
      \multicolumn{1}{c}{Parameters}   & \multicolumn{1}{c}{ Variations} \\
      \midrule
      \multirow{4}{*}{Turnover rate ($p_t$)}   & Stationary: [0.4, 0.4, 0.4, 0.4, 0.4, 0.4, 0.4, 0.4, 0.4, 0.4, 0.4, 0.4]\\
      & Increasing: [0.1, 0.1, 0.1, 0.3, 0.3, 0.3, 0.5, 0.5, 0.5, 0.7, 0.7, 0.7] \\
      & Decreasing: [0.7, 0.7, 0.7, 0.5, 0.5, 0.5, 0.3, 0.3, 0.3, 0.1, 0.1, 0.1]\\
      &Seasonal: [0.1, 0.3, 0.5, 0.7, 0.5, 0.3, 0.1, 0.3, 0.5, 0.7, 0.5, 0.3]\\
      \specialrule{0em}{2pt}{2pt}
      \multirow{10}{*}{Min required staff ($W_t$)}   &STA: [100, 100, 100, 100, 100, 100, 100, 100, 100, 100, 100, 100]\\ 
      & LCY1: [162, 166, 166, 162, 153, 141, 126, 110, 93, 77, 61, 48]\\
      & LCY2: [36, 48, 61, 77, 93, 110, 126, 141, 153, 162, 166, 166]\\
      & SIN1: [100, 43, 20, 43, 100, 157, 180, 157, 100, 43, 20, 43]\\
    & SIN2: [100, 79, 70, 79, 100, 121, 130, 121, 100, 79, 70, 79]\\
    & RAND: [91, 33, 79, 2, 76, 109, 115, 224, 22, 48, 136, 13]\\
    & EMP1: [51, 152, 467, 268, 489, 446, 248, 281, 363, 155, 293, 220]\\
    & EMP2: [81, 236, 394, 164, 287, 508, 391, 754, 694, 261, 195, 320]\\
    &EMP3: [116, 264, 144, 146, 198, 74, 183, 204, 114, 165, 318, 119]\\
    &EMP4: [188, 64, 279, 453, 224, 223, 517, 291, 547, 646, 224, 215]\\
\specialrule{0em}{2pt}{2pt}
    Fixed hiring cost ($K$) &2000, 4000, 6000\\
\specialrule{0em}{2pt}{2pt}
Salary ($h$) &1500, 2000, 3000\\
\specialrule{0em}{2pt}{2pt}
Penalty cost ($\pi$) &3500, 4000, 5000\\
      \bottomrule
    \end{tabular}
   
\end{table}


Table~\ref{table:NumericalResults} summarizes the numerical performance. Here, SDP is the  exact stochastic dynamic programming solution, MILP indicates the mixed-integer linear programming model of the static-dynamic control strategy presented in Section \ref{sec:MILP}, while MILP--$sS$ denotes the simulated results evaluated over 10,000 turnover samples, with $(s,S)$ parameters determined by the binary search algorithm in Section \ref{sec:ss_computation}. For performance assessment, we report 95\% confidence intervals for the optimality gap of MILP and MILP--$sS$, benchmarked against the exact stochastic dynamic programming solution, as well as the computational times of all the approaches. The optimality gap is computed as follows:
\[
\mathrm{Gap}
= \frac{|Z^{\mathrm{method}}-Z^{\mathrm{SDP}}|}
{Z^{\mathrm{SDP}}}\times 100\%,
\]
where \(Z^{\mathrm{method}}\) denotes the objective value obtained by MILP or the simulated objective value of MILP--$sS$, and \(Z^{\mathrm{SDP}}\) denotes the
objective value obtained by the exact SDP model.


Results show that both the MILP model and the MILP–$(s, S)$ policy achieve outstanding performance across all scenarios, with average optimality gaps as low as 3.03\% and 0.35\%, respectively. Both approaches maintain high computational efficiency, requiring an average of only 0.32 seconds for the MILP model and 5.36 seconds for the MILP–$(s, S)$ policy, compared with the average computational time 161.15s of SDP. Notably, parameter variations show no distinct monotonic effect on the performance of the MILP–$(s, S)$ policy; however, for the MILP, its optimality gap increases with the unit penalty cost and decreases with the salary.

\begin{table}[!ht]
\centering
\caption{Computational study results of different computation methods.}\label{table:NumericalResults}
\begin{tabular}{lcccccc}
\toprule
& SDP &  \multicolumn{2}{c}{MILP}  &
\multicolumn{2}{c}{MILP-sS} &\multirow{2}*{Cases}\\
\cmidrule(lr){3-4} \cmidrule(lr){5-6}
 & Time (s) & Opt. gap (\%) &Time (s)  &Opt. Gap (\%)  & Time (s)  \\
\midrule
Turnover rate \\
\specialrule{0em}{1pt}{1pt}
Decreasing & 162.19 $\pm$ 1.13 & 2.67 $\pm$ 0.18 & 0.40 $\pm$ 0.03 & 0.33 $\pm$ 0.03 & 7.11 $\pm$ 0.47 & 270 \\
Increasing & 156.34 $\pm$ 0.34 & 3.37 $\pm$ 0.22 & 0.26 $\pm$ 0.02 & 0.36 $\pm$ 0.04 & 3.80 $\pm$ 0.19 & 270 \\
Seasonal & 156.25 $\pm$ 0.34 & 3.20 $\pm$ 0.20 & 0.29 $\pm$ 0.02 & 0.35 $\pm$ 0.04 & 4.88 $\pm$ 0.27 & 270 \\
Stationary & 169.80 $\pm$ 1.31 & 2.90 $\pm$ 0.20 & 0.34 $\pm$ 0.03 & 0.38 $\pm$ 0.04 & 5.64 $\pm$ 0.29 & 270 \\
\specialrule{0em}{3pt}{3pt}
Min required staff\\
\specialrule{0em}{1pt}{1pt}
EMP1 & 158.32 $\pm$ 1.68 & 2.08 $\pm$ 0.17 & 0.31 $\pm$ 0.04 & 0.25 $\pm$ 0.03 & 5.34 $\pm$ 0.53 & 108 \\
EMP2 & 156.40 $\pm$ 1.67 & 1.73 $\pm$ 0.14 & 0.30 $\pm$ 0.04 & 0.22 $\pm$ 0.03 & 5.45 $\pm$ 0.52 & 108 \\
EMP3 & 161.35 $\pm$ 1.77 & 2.31 $\pm$ 0.18 & 0.35 $\pm$ 0.05 & 0.38 $\pm$ 0.04 & 6.15 $\pm$ 0.57 & 108 \\
EMP4 & 156.80 $\pm$ 1.72 & 3.02 $\pm$ 0.51 & 0.36 $\pm$ 0.04 & 0.26 $\pm$ 0.03 & 5.66 $\pm$ 0.59 & 108 \\
LCY1 & 162.76 $\pm$ 1.72 & 3.62 $\pm$ 0.26 & 0.25 $\pm$ 0.04 & 0.34 $\pm$ 0.07 & 5.24 $\pm$ 0.50 & 108 \\
LCY2 & 162.68 $\pm$ 1.72 & 3.83 $\pm$ 0.28 & 0.36 $\pm$ 0.04 & 0.36 $\pm$ 0.08 & 5.00 $\pm$ 0.43 & 108 \\
RAND & 163.10 $\pm$ 1.65 & 2.36 $\pm$ 0.24 & 0.42 $\pm$ 0.06 & 0.47 $\pm$ 0.05 & 5.41 $\pm$ 0.66 & 108 \\
SIN1 & 163.90 $\pm$ 1.70 & 3.23 $\pm$ 0.24 & 0.25 $\pm$ 0.03 & 0.39 $\pm$ 0.05 & 4.21 $\pm$ 0.62 & 108 \\
SIN2 & 163.04 $\pm$ 1.62 & 4.06 $\pm$ 0.29 & 0.30 $\pm$ 0.04 & 0.44 $\pm$ 0.09 & 5.34 $\pm$ 0.50 & 108 \\
STA & 163.12 $\pm$ 1.70 & 4.11 $\pm$ 0.30 & 0.33 $\pm$ 0.04 & 0.40 $\pm$ 0.08 & 5.77 $\pm$ 0.54 & 108 \\

\specialrule{0em}{3pt}{3pt}
Fixed cost \\
\specialrule{0em}{1pt}{1pt}
2000 & 158.07 $\pm$ 0.57 & 3.06 $\pm$ 0.18 & 0.32 $\pm$ 0.02 & 0.36 $\pm$ 0.04 & 5.25 $\pm$ 0.29 & 360 \\
4000 & 163.44 $\pm$ 1.13 & 3.04 $\pm$ 0.17 & 0.32 $\pm$ 0.02 & 0.35 $\pm$ 0.03 & 5.36 $\pm$ 0.30 & 360 \\
6000 & 161.93 $\pm$ 1.01 & 3.01 $\pm$ 0.17 & 0.33 $\pm$ 0.02 & 0.34 $\pm$ 0.03 & 5.46 $\pm$ 0.32 & 360 \\
\specialrule{0em}{3pt}{3pt}
Salary\\
\specialrule{0em}{1pt}{1pt}
1500 & 161.30 $\pm$ 0.96 & 4.27 $\pm$ 0.16 & 0.19 $\pm$ 0.01 & 0.23 $\pm$ 0.02 & 3.73 $\pm$ 0.08 & 360 \\
2000 & 161.05 $\pm$ 0.97 & 3.20 $\pm$ 0.13 & 0.24 $\pm$ 0.01 & 0.22 $\pm$ 0.02 & 4.42 $\pm$ 0.13 & 360 \\
3000 & 161.09 $\pm$ 0.98 & 1.63 $\pm$ 0.09 & 0.54 $\pm$ 0.03 & 0.61 $\pm$ 0.04 & 7.92 $\pm$ 0.38 & 360 \\
\specialrule{0em}{3pt}{3pt}
Penalty cost\\
\specialrule{0em}{1pt}{1pt}
3500 & 161.08 $\pm$ 0.96 & 2.36 $\pm$ 0.15 & 0.43 $\pm$ 0.03 & 0.45 $\pm$ 0.05 & 6.76 $\pm$ 0.40 & 360 \\
4000 & 161.40 $\pm$ 0.99 & 2.94 $\pm$ 0.16 & 0.31 $\pm$ 0.02 & 0.32 $\pm$ 0.03 & 5.18 $\pm$ 0.24 & 360 \\
5000 & 160.96 $\pm$ 0.95 & 3.81 $\pm$ 0.18 & 0.22 $\pm$ 0.01 & 0.29 $\pm$ 0.02 & 4.13 $\pm$ 0.14 & 360 \\


\specialrule{0em}{3pt}{3pt}
General & 161.15 $\pm$ 0.56 & 3.03 $\pm$ 0.10 & 0.32 $\pm$ 0.01 & 0.35 $\pm$ 0.02 & 5.36 $\pm$ 0.18 & 1080 \\
\bottomrule
\end{tabular}
\end{table}

\section{Conclusions}\label{sec:conclusion}

This paper studied a finite-horizon workforce planning problem with fixed recruitment costs and stochastic staff turnover. The distinctive feature of the model is that turnover is decision-dependent: after recruitment, the number of potential leavers depends on the post-hiring workforce level, and departures follow a binomial distribution. This feature prevents a direct application of the classical inventory-control argument for \((s,S)\) optimality, because the transition probabilities vary with the hiring decision.

We showed that this difficulty can be overcome under binomial turnover. First, we established the discrete convexity of the single-period expected salary and shortage-penalty cost. We then  prove that \(K\)-convexity is preserved under the binomial propagation induced by workforce turnover. These results imply that the optimal hiring policy has an \((s,S)\) structure in each period: recruitment is initiated only when the workforce level is sufficiently low, and staff are then hired up to a target level. This provides a structural explanation for lumpy hiring in settings with stochastic turnover and fixed recruitment costs, consistent with empirical evidence on discontinuous employment adjustment \citep[e.g.,][]{hamermesh1989labor,caballero1997aggregate}.

We also developed a tractable mixed-integer linear programming approximation for computing the policy parameters. Across our test bed with only two sub-gradient segments, the MILP approximation achieved an average optimality gap of 3.03\% and an average computation time of 0.32 seconds. The policy obtained from the MILP-based \((s,S)\) computation performed even better in simulation, with an average optimality gap of 0.35\%, although at a higher average computation time of 5.36 seconds. These results indicate that the structural policy can be computed efficiently and that the proposed approximation provides a practical way to implement the \((s,S)\) hiring rule.

Several directions remain open for future research. Natural extensions include richer multi-grade workforce systems, grade-dependent turnover and promotion behaviour, recruitment failure, and hiring capacity constraints. Another promising direction is to study how the \((s,S)\) structure changes when turnover probabilities depend on tenure, workload, or broader labour-market conditions. Finally, one may explore recursion-free approximations similar to \cite{kilic2024simple} when the variable hiring cost is absent.







\printbibliography[title={References}]

\clearpage

\begin{appendices}
\renewcommand\sectionname{Appendix}

\section{Proof for Theorem \ref{theorem-K-convex}}\label{app:theorem-optimality}

\begin{proof}
We prove the result by backward induction. Throughout the proof, all functions are defined on
\(\mathbb{N}\), and the notion of \(K\)-convexity is understood on this discrete domain.

First, we show that if \(G_t\) is \(K\)-convex, then the optimal policy in period \(t\) has an \((s,S)\) structure and the corresponding value function \(C_t\) is also
\(K\)-convex. By
Lemma~\ref{lemma:K_properties}\ref{item-Kconvexity:d}, since \(G_t\) is \(K\)-convex, $C_{t+1}\geq 0$, and $G_t(y)\ge cy+L(y)\rightarrow \infty$ when $y\rightarrow \infty$, there exist
thresholds \(s_t\) and \(S_t\), with \(S_t\) being a global minimizer of \(G_t\), such that
\[
    G_t(S_t) \leq G_t(x), \qquad \forall x\in \mathbb{N},
\]
and
\[
    G_t(s_t) \leq G_t(S_t)+K.
\]
Moreover, for any \(s_t \leq x \leq z\),
\[
    G_t(x) \leq G_t(z)+K.
\]
For notational simplicity, we suppress the subscript \(t\) and write \(s\), \(S\), and \(G\).

Given a pre-hiring workforce level \(x\), either no hiring is made, in which case the cost is
\(G(x)-cx\), or a positive hiring quantity is chosen, in which case the fixed cost is incurred and
the optimal hire-up-to level is \(S\). Following the usual convention that no hiring is chosen at the boundary when the two actions are tied,
we can write
\begin{equation}
C_t(x) =
\begin{cases}
G(S)+K-cx, & x<s,\\
G(x)-cx, & x\geq s.
\end{cases}
\label{eq:CxGx_appendix}
\end{equation}


The above formula implies that the optimal policy is $(s, S)$. We now prove that \(C_t\) is \(K\)-convex. It suffices to verify that, for all \(x,a,b\in\mathbb{N}\),
\begin{equation}
    b\left(K+C_t(x+a+b)-C_t(x)\right)
    \geq
    (a+b)\left(C_t(x+b)-C_t(x)\right).
\label{ieq:Kconvex_appendix}
\end{equation}

The cases \(a=0\) or \(b=0\) are immediate, so assume \(a,b>0\). We distinguish four cases according
to the positions of \(x\), \(x+b\), and \(x+a+b\) relative to \(s\).

\paragraph{Case 1: \(x\geq s\).}
Then \(x\), \(x+b\), and \(x+a+b\) all lie in the region \(z\geq s\). By
\eqref{eq:CxGx_appendix},
\[
    C_t(z)=G(z)-cz
    \qquad \text{for all } z\in\{x,x+b,x+b+a\}.
\]
Since subtracting a linear function preserves \(K\)-convexity, and \(G\) is \(K\)-convex, inequality
\eqref{ieq:Kconvex_appendix} follows directly.

\paragraph{Case 2: \(x<x+b<s\) and \(x+a+b\geq s\).}
In this case,
\[
    C_t(x)=G(S)+K-cx,\qquad
    C_t(x+b)=G(S)+K-c(x+b),
\]
and
\[
    C_t(x+a+b)=G(x+a+b)-c(x+a+b).
\]
Substituting these expressions into \eqref{ieq:Kconvex_appendix}, the terms involving the linear cost
cancel, and the desired inequality reduces to
\[
    G(x+a+b)-G(S)\geq 0.
\]
This holds because \(S\) is a global minimizer of \(G\).

\paragraph{Case 3: \(x<s\leq x+b<x+a+b\).}
Here,
\[
    C_t(x+b)=G(x+b)-c(x+b),\qquad
    C_t(x+a+b)=G(x+a+b)-c(x+a+b),
\]
whereas
\[
    C_t(x)=G(S)+K-cx.
\]
After substituting these expressions into \eqref{ieq:Kconvex_appendix} and cancelling the linear
terms, it remains to prove
\begin{equation}
    b\left(G(x+a+b)-G(S)\right)
    \geq
    (a+b)\left(G(x+b)-G(S)-K\right).
\label{ieq:case3_target}
\end{equation}
Since \(G\) is \(K\)-convex, applying the definition with three points \(s\), \(x+b\) and \(x+a+b\) gives
\[
    (x+b-s)\left(K+G(x+a+b)-G(s)\right)
    \geq
    (x+a+b-s)\left(G(x+b)-G(s)\right).
\]

Define $\mathcal{H}\triangleq G(S)+K-G(s)$. $\mathcal{H}\geq 0$ by the definition of $s$. Therefore,
\[
    K+G(x+a+b)-G(s)=G(x+a+b)-G(S)+\mathcal{H},
\]
and
\[
    G(x+b)-G(s)
    =
    G(x+b)-G(S)+\mathcal{H}-K.
\]
It follows that
\[
(x+b-s)\left(G(x+a+b)+\mathcal{H}-G(S)\right)\geq
    (x+a+b-s)\left(G(x+b)-G(S)-K+\mathcal{H}\right),
\]
and 
\begin{align*}
(x+b-s)\left(G(x+a+b)-G(S)\right)
    \geq&
    (x+a+b-s)\left(G(x+b)-G(S)-K\right)+a\mathcal{H}\\
    \ge &(x+a+b-s)\left(G(x+b)-G(S)-K\right).
\end{align*}

Since \(x<s\leq x+b<x+a+b\), we have

$$ 0\le\frac{x+b-s}{x+a+b-s}\le\frac{b}{a+b}.$$

$G(x+a+b)-G(S)\geq 0$ by the definition of $S$. Replacing $\frac{x+b-s}{x+a+b-s}$ with $\frac{b}{a+b}$ in the preceding inequality, this yields:
\(
b\left(G(x+a+b)-G(S)\right)\geq (a+b)\left(G(x+b)-G(S)-K\right),
\)
which is exactly \eqref{ieq:case3_target}. Hence \eqref{ieq:Kconvex_appendix} holds in this case.

\paragraph{Case 4: \(x+a+b<s\).}
Then \(x\), \(x+b\), and \(x+a+b\) all lie below \(s\). By \eqref{eq:CxGx_appendix},
\[
    C_t(z)=G(S)+K-cz
\]
on this range. Thus \(C_t\) is linear over the relevant points, and
\eqref{ieq:Kconvex_appendix} holds immediately.

Combining the four cases, we conclude that \(C_t\) is \(K\)-convex whenever \(G_t\) is \(K\)-convex.

It remains to establish the induction. In the terminal period \(T\), we have
\[
    G_T(x)=cx+L_T(x).
\]
The function \(L_T(x)\) is discretely convex by Lemma ~\ref{lemma:gLconvex}. Hence it is
\(0\)-convex, and therefore \(K\)-convex by
Lemma~\ref{lemma:K_properties}\ref{item-Kconvexity:a}. The linear term \(cx\) is also
\(0\)-convex. By Lemma~\ref{lemma:K_properties}\ref{item-Kconvexity:b}, \(G_T\) is \(K\)-convex.
The argument above then implies that \(C_T\) is \(K\)-convex and that the optimal policy in period
\(T\) is of \((s,S)\) type.

Now suppose that \(C_{t+1}\) is \(K\)-convex for some \(t<T\). By the definition of \(G_t\),
\[
    G_t(y)
    =
    cy+L_t(y)+
    \sum_{k=0}^{y} C_{t+1}(y-k) f_y(k).
\]
Equivalently,
\[
    G_t(y)
    =
    cy+L_t(y)+\mathbb{E}\!\left[C_{t+1}(y-\omega_y)\right],
    \qquad
    \omega_y\sim \mathrm{Binomial}(y,p_t).
\]
The term \(cy\) is linear and hence \(0\)-convex. The function \(L_t(y)\) is discretely convex by
Lemma ~\ref{lemma:gLconvex}, and hence \(0\)-convex. Since \(C_{t+1}\) is \(K\)-convex, Lemma
\ref{lemma:K_properties}\ref{item-Kconvexity:c} implies that
\(\mathbb{E}\!\left[C_{t+1}(y-\omega_y)\right]
\) is \(K\)-convex. Therefore, by Lemma~\ref{lemma:K_properties}\ref{item-Kconvexity:b}, \(G_t\) is
\(K\)-convex. Applying the first part of the proof again, \(C_t\) is \(K\)-convex and the optimal
policy in period \(t\) has an \((s,S)\) structure.

By backward induction, \(G_t\) and \(C_t\) are \(K\)-convex for all
\(t=1,\ldots,T\), and the optimal hiring policy in each period is an \((s,S)\) policy.
\end{proof}

\section{Proof of Lemma \ref{prop:forward-diff}}\label{sec:lemmaForwardDiff}

\begin{proof}
For the sake of simplicity let us discard the subscript from $g_{ij}$ and work in this Lemma directly with $g(\cdot).$

From the proof of Lemma ~\ref{lemma:gLconvex}, the forward difference of $g(y)$ over a single period can be written as
\[
g(y+1)-g(y)=(1-p)\bigl(H(1)-H(0)\bigr).
\]

In a hiring cycle from period $j$ to period $t$, the probability that turn over occurs throughout the cycle is
$1-q_{jt}=1-\prod_{k=j}^t(1-p_k),$ so $(1-p)=q_{jt}$ as defined in \ref{def:qjt}. Hence, by interpreting $g(y)$ as the expected number of outsourced staff at the end of the hiring cycle, we obtain
\begin{align*}
g(y+1)-g(y)
&=q_{jt}\bigl(H(1)-H(0)\bigr)\\
&=q_{jt}\left(\mathbb{E}[\omega^y-y-1+W]^+-\mathbb{E}[\omega^y-y+W]^+\right).
\end{align*}

Define the random variable
\[
X :=  \omega^y-y+W.
\]
Then the difference of the two expectations can be rewritten as
\[
\mathbb{E}[\omega^y-y-1+W]^+-\mathbb{E}[\omega^y-y+W]^+
=\mathbb{E}\big[(X-1)^+-X^+\big].
\]

For any integer number $x$, the positive-part operator satisfies
\[
(x-1)^+-x^+=
\begin{cases}
-1, & x>0,\\
0, & x\le 0,
\end{cases}
= -1+\mathbf{1}\{x\le 0\}.
\]
Applying this identity pointwise to the random variable $X$ and taking expectations yields
\[
\mathbb{E}\big[(X-1)^+-X^+\big]
=-1+\mathbb{E}\big[\mathbf{1}\{X\le 0\}\big]
=-1+\mathbb{P}(X\le 0).
\]

Since $X=\omega^y-y+W$, the event $\{X\le 0\}$ is equivalent to $\omega^y\le y-W$. Therefore,
\[
g(y+1)-g(y)
=q_{jt}\bigl(-1+\mathbb{P}(\omega^y\le y-W)\bigr)=-q_{jt}(1-F^y(y-W)),
\]
and this completes the proof.
\end{proof}

\end{appendices}

\end{document}